\documentclass[a4paper,final, 12pt]{amsart}
\usepackage[foot]{amsaddr} 

\usepackage[T1]{fontenc}
\usepackage{amsmath,amssymb,amscd,amsfonts,amsthm,mathrsfs,eucal,dsfont,verbatim}

\usepackage{lscape,enumerate}
\usepackage[margin=2in]{geometry}
\usepackage{hyperref}

\newcommand{\real}{\mathds{R}}
\newcommand{\XX}{\mathbb{X}}

\newcommand{\Ee}{\mathds E}
\newcommand{\Pp}{\mathds P}
\newcommand{\I}{\mathds 1}

\newcommand{\Ff}{\mathcal{F}}

\newcommand{\LL}{\mathcal{L}}
\newcommand{\Bb}{\mathcal{B}}
\def\1{1\!\!\hbox{{\rm I}}}

\newcommand{\sign}{\operatorname{sign}}

\usepackage{color}
\usepackage{xcolor}

\newcommand{\normal}{\color{black}}

\evensidemargin
-1cm

\usepackage[T2A]{fontenc}

\numberwithin{equation}{section}
\theoremstyle{plain}
\newtheorem{theorem}{Theorem}
\newtheorem{lemma}[theorem]{Lemma}

\newtheorem{corollary}[theorem]{Corollary}
\theoremstyle{definition}
\newtheorem{remark}[theorem]{Remark}
\newtheorem{definition}{Definition}

\numberwithin{theorem}{section}
\numberwithin{definition}{section}

\usepackage{etoolbox}

\makeatletter
\newcommand{\myfnsymbol}[1]{\ifcase#1
	\or\star
	\or\dagger
	\fi
}
\makeatother
\let\oldmaketitle\maketitle
\renewcommand{\maketitle}{\oldmaketitle\setcounter{footnote}{0}}

\begin{document}

	\title{On ergodic properties of some L\'evy-type processes}
	
	\date{}

	\author[V. Knopova]{Victoria Knopova}
	\address[V. Knopova]{Kiev T. Shevchenko University\\ Department of Mechanics and Mathematics\\ Acad. Glushkov Ave., 02000,  Kiev, Ukraine\\ \texttt{Email: $vicknopova$@$knu.ua$ (corresponding author)}}

		\author[Y. Mokanu]{Yana Mokanu}
	\address[Y. Mokanu]{Kiev T. Shevchenko University\\ Department of Mechanics and Mathematics\\ Acad. Glushkov Ave., 02000,  Kiev, Ukraine\\
	\texttt{Email:  $yana.mokanu3075$@$gmail.com$} }
	
	\begin{abstract}
		In this note we prove some sufficient conditions for ergodicity of a L\'evy-type process,  such  that on the test functions  the generator of the respective semigroup  is of the form 
		$$
		Lf(x) = a(x)f'(x) + \int_\real \left( f(x+u)-f(x)- \nabla f(x)\cdot  u \I_{|u|\leq 1} \right) \nu(x,du), \quad f\in C_\infty^2(\real). 
		$$
		Here $\nu(x,du)$ is a L\'evy-type kernel  and $a(\cdot): \real\to \real$. We consider the case when  the tails are of polynomial decay as  well as the case when the decay is (sub)-exponential.
		For the proof  the  Foster-Lyapunov approach is used. 
	\end{abstract}
	
	

	\maketitle

\section{Introduction}\label{intro}
	Consider an operator 
	\begin{equation}\label{L1}
		Lf(x) = a(x) f'(x) + \int_\real \left( f(x+u)-f(x)- \nabla f(x)\cdot  u \I_{|u|\leq 1} \right) \nu(x,du), \quad f\in C_0^2(\real), 
	\end{equation}
	where $C_0^2(\real)$ is the class of twice continuously differentiable functions with compact support, $a(\cdot): \real \to \real$  and  the kernel  $\nu(x,du)$ satisfies 
	\begin{equation}\label{bdd}
\sup_x  \int_\real (1\wedge |u|^2) \nu(x,du)<\infty. 
	\end{equation}
Under certain conditions (see Section~\ref{Settings}  below) one can show that the martingale problem for $(L, C_0^2(\real))$ is well posed and its solution is a Markov process possessing the Feller property (i.e. the respective semigroup of operators $(P_t)_{t\geq 0}$, $P_t f(x) := \Ee^x f(X_t)$, $f\in C_\infty(\real)$, preserves the space $C_\infty(\real)$ of continuous functions vanishing at infinity).  The respective process $X$ is often called a \emph{L\'evy-type} process, because locally its structure resembles the structure of a L\'evy process, e.g. a  process with independent  stationary  increments. 

The aim of this note is to investigate the ergodicity in total variation  of the process $X_t$ related to \eqref{L1} and to describe the speed of convergence to the respective invariant measure.  
In order to investigate this problem, we need a certain set up.  Following the framework of \cite{Ku17}, if the transition probability kernel satisfies the local Dorbushin condition and if the so-called extended generator of $(P_t)_{t\geq 0}$  satisfies the Lyapunov type   condition  with some Lyapunov function,  then one can reduce the continuous time  setting to the discrete one (cf. \cite[Th.2.8.9, Th.3.2.3]{Ku17}), and apply the ergodic theorems for the skeleton chain  $X^h := \{ X_{nh}, \, n\in \mathbb{Z}_+\}$.  Since the ergodicity of the skeleton chain implies the ergodicity of the initial Markov process with the same rates, we derive the desired statement.  
This problem in the context of  L\'evy-driven SDEs is investigated in \cite{Ku17}, see also \cite{KP19}. On the other hand, not any L\'evy-type process can be derived easily as a solution to some L\'evy-driven SDE, but nevertheless such process are of interests, which motivated the current research. Knowing  the speed of convergence to the ergodic distribution allows us to use simulation methods (see, for example, \cite{B10})   in  order to approximate this distribution in a ``long run''  having the distribution of the initial L\'evy-type process.   

In this note we consider for simplicity the one-dimensional situation.  We believe that under certain control on the distortion in different directions of the L\'evy kernel similar result should  hold in the multi-dimensional settings (see \cite{Ku17} for the respective results on the L\'evy-driven SDEs). However, when the L\'evy type kernel $\nu(x,du)$ is not regular enough, the choice of the Lyapunov function and the proof of the Lyapunov type inequality  is not so obvious. Therefore we postpone this  problem for further investigations.  

The paper is organized as follows.  In Section~\ref{Settings}  we recall the notion of ergodicity and  discuss the assumptions.   Then we give the definition of the full generator and explain which functions are in its domain.   Further, we recall  the  Dobrushin condition and the definitions of  the  Lyapunov function and  of the  Lyapunov type equation. This is the framework necessary for the investigation of  ergodicity.  In Section~\ref{Results} we formulate our results, e.g.  we provide the conditions under which our process is ergodic and get the speed of convergence to the invariant  measure. The speed of convergence heavily relies on the order of the moments of the kernel $\nu(x,du)$, therefore we consider two types of conditions: when we have polynomial decay of the tails  and the ``(sub)-exponential'' decay.  Proofs are given in Section~\ref{Proofs}.

\section{Settings}\label{Settings}

In order to make the paper self-contained, we provide below the necessary notions and give some overview of the existing results. 

Recall that a Markov process $(X_t)_{t\geq 0}$  on $\real$  with transition probability kernel $P_t(x,dy)$ is called
\emph{ergodic} if  there exists an  invariant probability measure $\pi(\cdot)$ such that 
\begin{equation}\label{inv1}
	\lim_{t\to \infty} \|\Pp_t(x,\cdot)- \pi(\cdot)\|_{TV} =0
\end{equation}
for any $x\in \real$. 
Here $\|\cdot\|_{TV}$ is the total variation norm.

There are classical results   on Markov chains (cf. \cite[Ch.2]{Ku17}) establishing ergodicity under  some additional assumptions.  As we mentioned in Section~\ref{intro},  having a continuous time Markov process one can construct the so-called \emph{skeleton chain}, establish the ergodicity for this chain,  which in turn implies the ergodicity of the initial Markov process with the same convergence rates. Therefore, we formulate a few definitions and  results for a Markov  chain and then indicate the ``return'' route to the initial  process.  

Let $X$ be a Markov chain on a state metric space $(\XX,\rho)$, where $\rho$ is the metric on $\XX$,  with transition kernel $\Pp(x,dy)$.

The first condition we need for the ergodicity is  a \emph{local Dobrushin condition}  for $X$. 
We say that a chain $X$ satisfies the local Dobrushin condition (cf. \cite[p. 42]{Ku17}) on a measurable set $B\subset \XX\times \XX$  if 
\begin{equation}\label{Dor1}
	\sup_{x,y\in B} \|\Pp(x,\cdot)- \Pp(y,\cdot)\|_{TV} <2. 
\end{equation}

Condition \eqref{Dor1} is not always  easy to check, but there are sufficient conditions which are much more clear. Suppose that the probability kernel $\Pp(x,dy)$ is of the form 
$$
\Pp(x,dy) = \Pp_c(x,dy) + \Pp_d (x,dy), 
$$
where 
\begin{itemize}
	\item[i)]  $\Pp_c(x,dy)$, $\Pp_d (x,dy)$ are non-negative kernels; 
	\item[ii)] there exists a measure $m$ on $\XX$ such that 
	$$
	\Pp_c(x,dy)\ll m(dy); 
	$$
	\item[iii)]  for a given point $x^* \in \XX$ the mapping 
	$$
	x\mapsto \frac{\Pp_c(x,dy)}{m(dy)}\in L_1(\XX,m)
	$$
	is continuous at $x^*$ and 
	$$
	\Pp_c(x^*,\XX)>0. 
	$$
\end{itemize}
Assumptions i)--ii) are sufficient for the Dobrushin  condition \eqref{Dor1} to hold true, see \cite[Prop.2.9.1]{Ku17}.

The second   condition we need  is the \emph{Lyapunov type} condition. 
Suppose that there exist 

\begin{itemize}
	\item   a  norm-like function $V: \, \XX\to [1,\infty)$ (i.e. $V(x)\to \infty$ as $|x|\to  \infty$),  $V$ is bounded on a compact $K$, 
	
	\item  some function $f: \, [1,\infty) \to (0,\infty)$, which  admits a non-negative increasing and concave extension to $[0,\infty)$, 
	
	\item  a constant $C>0$,
\end{itemize}
  such that the following  relations  hold: 
\begin{equation}\label{Lyap1}
	\Ee_x V(X_1) - V(x) \leq - f (V(x))+ C, \quad x\in \XX, 
\end{equation}
\begin{equation}\label{Lyap11}
f\left(1+ \inf_{x\notin K} V(x)\right) > 2C.  
\end{equation}
Then (cf. \cite[Th.2.8.6]{Ku17}) the chain $X$ satisfies the so-called \emph{coupling condition}, which is the second counterpart for the required ergodicity. Such a function $V$ is called a \emph{Lyapunov function}.

Dobrushin and Lyapunov conditions  allow to prove that the chain $X$ is ergodic  (cf. \cite[Th.2.7.2]{Ku17}), and  find the convergence rates. Namely, the following result holds true (cf. \cite[2.8.8]{Ku17}): The invariant probability measure $\pi$ is unique, satisfies 
$$
\int_\real f(V(x))\mu(dx)<\infty, 
$$
and for any $\delta\in (0,1)$  there exist $c,\gamma>0$ such that 
\begin{equation}\label{erg}
	\| \Pp_n (x,\cdot) - \pi(\cdot)\|_{TV} \leq \psi^\delta (n) \left( f^\delta (V(x))+ \int_\real f(V(x))\mu(dx) \right). 
\end{equation}
Here 
\begin{equation}\label{F}
  \psi(n) :=  \frac{1}{f(F^{-1}(\gamma n))}, \quad 	F(t) := \int_1^t \frac{dw}{f(w)}. 
\end{equation}

Therefore, in order to follow  the strategy described above we need  to find the assumptions   on the continuous time Markov process $X_t$ related to \eqref{L1}, which  guarantee \eqref{Lyap1} and \eqref{Lyap11}.

Let us come back to the operator defined by \eqref{L1}.  The problem how to relate an  operator  \eqref{L1} and a (strong) Feller process was investigated a lot; see  \cite{Ja01}, \cite{Ja02},   \cite{G68}, \cite{Ko73}, \cite{Ko84a}, \cite{MP92a}, \cite{MP92b}, \cite{Ja92}, \cite{Ja94}, \cite{JL93}, \cite{Ne94},  \cite{KN97}, \cite{Ho98}, \cite{Ho00}, \cite{Ku19},  \cite{BKS20}, \cite{KK17}--\cite{KKS20}.  
In   \cite{KKS20} we give  an extensive overview of this problem. For example, 
in the series of papers \cite{KK17}--\cite{KKS20}, \cite{BKS20},  it is shown (for ``bounded coefficients'')  that  under certain regularity and boundedness assumptions  on the kernel $\nu(x,du)$ the operator admitting the representation \eqref{L1}  extends to the generator of a strong Markov process $X=(X_t)_{t\geq 0}$, which has a  strictly positive transition probability density, continuous in $x$. Hence, the assumptions i)--iii) are satisfied, and the respective transition probability kernel $P_t (x,dy)$ satisfies the Dobrishin condition.  The case of unbounded drift is more complicated. In the case of a L\'evy-driven SDE we refer to \cite[Prop.3.4.7]{Ku17} for the Dobrushin condition; see also the recent paper \cite{MZ22} (and the references therein) where the  transition probability density of the process is constructed and the Dobrushin condition follows.   Unfortunately, at the moment we do not know the reference to the result, which allows effectively check the  Dobrushin condition for a general L\'evy-type process.

At this point we would like to comment    on the other standard references for the ergodicity conditions, e.g.  on the  works \cite{MT93a}--\cite{MT93d}. 
Namely, in this framework it is often required that the process $X_t$ is the \emph{Lebesgue irreducible $T$-model}, see \cite[Ch.6]{MT93c}, \cite[Th.7.1]{T94},  also \cite[Cor.6.26]{BSW13}. This assumption is satisfied, if the process $X_t$ admits a strictly   positive  transition probability density and the respective semigroup possesses a $C_b$--Feller property. By the $C_b$-Feller property we mean that the respective semigroup $(T_t)_{t\geq 0}$ maps  the space $C_b(\real)$ of bounded continuous functions  into $C_b(\real)$.  Note that if  $X$ is a Feller process and $T_t \I_B\in C_b$, then $X$ has the $C_b$-Feller property, see \cite[Th.1.9]{BSW13}. This property follows, for example, from the results, proved in \cite{KK17}--\cite{KKS20}, \cite{BKS20}, but again we account the problem how to show the existence of a positive continuous density in the case when the drift coefficient is unbounded.  

\medskip 

\emph{Through the paper we assume  that $a(x)$ and $\nu(x,du)$ are such that the martingale problem for $(\LL, D_0)$ is well posed, and the respective Markov Feller process such that for any $t>0$  the transition kernel  $P_t(x,dy)$ satisfies the Dobrushin condition. } 

\medskip 

Let us discuss now the Lyapunov type condition for the initial process. Our aim is to construct the Lyapunov function such that 
\eqref{Lyap1} holds true. Since we start with a Markov process related to \eqref{L1}, the inequality  we check is not actually \eqref{Lyap1}, but 
\begin{equation}\label{Lyap2}
\LL  V(x) \leq - f(V(x))  + C, 
\end{equation}
where $V$ and $f$ have the same meaning as above,   
and $\LL$ is the \emph{full generator} of $(T_t)_{t\geq 0}$. 
Recall the ``stochastic version'' of a \emph{full generator}, see \cite[(1.50)]{BSW13}. 

\begin{definition} The full generator of a measurable contraction semigroup $(T_t)_{t\geq 0}$  on $C_\infty(\real)$ is 
	\begin{equation}\label{ah}
		\begin{split}
			\hat{L}:= &\Big\{ (f,g)\in(\Bb(\real),\Bb(\real)) :\quad   M_t^{[f,g]}:=f(X_t)-f(x)= \int_0^t g(X_s)ds,\quad x\in \real,\, t\geq0,\Big.\\
			& \Big.\text{is a local  $\Ff_t^X:= \sigma(X_s, s\leq t)$ martingale} \Big\}.
		\end{split}
	\end{equation}
\end{definition}

Since  for $f\in D(L)$ 
\begin{equation}\label{gen1}
	T_t f(x) -f(x) = \int_0^t T_s Lf(x) ds \quad \forall t\geq 0, \,x\in \real,
\end{equation}
see \cite[Ch.1, Prop.1.5]{EK86}, then by the strong Markov property and the Dynkin formula $(f,Lf)\subset \hat{L}$ for any $f\in D(L)$.
Since for $f\in C^2(\real)$
$$
\left|\int_{|u|\leq 1} \left( f(x+u)- f(x) - f'(x) u\right) \nu(x,du) \right| \leq \max_{|u|\leq 1} |f''(x+u)| \int_{|u|\leq 1} u^2 \nu(x,du) <\infty,
$$
we have $Lf(x)<\infty$ for $f\in \mathcal{D}_0$,  where 
\begin{equation}\label{D2}
	\mathcal{D}_0:=\left\{ f\in C^2(\real): \quad\left| \int_{|u|\geq 1} (f(x+u)-f(x))\nu(x,du)\right|<\infty \quad \forall x\in \real\right\}. 
\end{equation}
 Using the localization procedure,  Ito formula  and taking the expectation,  we derive that in this $(f,Lf)\in \hat{L}$.  
We denote this extension of $L$ to $\mathcal{D}_0$ by $\LL$, and denote also by $\LL_0$ the extension of the integral part of $L$ to $\mathcal{D}_0$.

Thus, if we manage to prove \eqref{Lyap2}, then by \cite[Th.3.2.3]{Ku17} \eqref{Lyap1} holds true for the skeleton chain, and thus we have \eqref{erg}. Moreover, the same convergence rate will still hold for the initial Markov process:
\begin{equation}\label{erg2}
	\|P_t(x,\cdot)-\pi(\cdot)\|_{TV} \leq  C_{erg}(x) \psi^\delta(t), \quad t\to \infty,
\end{equation}
where  
\begin{equation}\label{F2}
\psi(t):= \frac{1}{ f(F^{-1}(\gamma t )))},
\end{equation}
 $f$ is the function for which the Lyapunov condition \eqref{Lyap2}  holds true,  $F$ is defined by \eqref{F}, and $C_{erg}(x)$ is the respective constant (depending on $x$). In this case we call the initial process $X$ \emph{$\psi$-ergodic. }

Let us briefly discuss some other existing results.  Ergodicity of Markov processes is studied a lot and under various  assumptions. We quote here only  several results which are in a similar context.     The results in some sense similar to ours were obtained in Wang   \cite{W08},  see also \cite{W13}, and in the papers of Sandric \cite{Sa13}, \cite{Sa16a}. 
See  Barky,  Cattiaux,  Guillin \cite{BCG08} for the approach  relying on  functional  inequalities, and 
Douc,  Fort, Guillin \cite{DFG09} for the condition of  positive Harris recurrence and  ergodicity, which in turn relies on   Fort, Roberts \cite{FR05}, 
and  Douc, Fort,  Moulines,  Soulier \cite{DFMS04} for the case  Markov chains, see also \cite{DFMS18}.   

\section{Results}\label{Results}
 
Through the paper we assume that
\begin{equation}\tag{\bfseries S}\label{S}
		\nu(x,\cdot) \quad \text{ is  symmetric in the sense that $\nu(x,A)= \nu(x,-A)$,   $ A\in\mathcal{B}(\real)$, $\{0\}\notin A$}. 
\end{equation} 	
	Denote by $N(x,u)$ the tails of $\nu$: 
	\begin{equation}\label{n22}
		N(x,u) :=	\nu(x,(-\infty,u]), \quad  u<0. 
	\end{equation}
We consider three cases. First we assume that the tails of $\nu$ have polynomial decay and 
	\begin{equation}\tag{\bfseries N1}\label{N1}
		\lambda^{-\delta(z)}\leq  \liminf_{x\to \infty} \frac{N(z,\lambda x)}{N(z,x)}
		\leq \limsup_{x\to \infty} \frac{N(z,\lambda x)}{N(z,x)} \leq \lambda^{-\sigma(z)} \quad \forall \lambda\geq 1,
	\end{equation}
	for some $0<\sigma\leq \sigma(z)\leq \delta(z)\leq \delta <\infty$.
	Note  that this relation  is equivalent to 
		\begin{equation}\tag{\bfseries N1'}\label{N1'}
		\lambda^{-\sigma(z)}\leq  \liminf_{x\to \infty} \frac{N(z,\lambda x)}{N(z,x)}
		\leq \limsup_{x\to \infty} \frac{N(z,\lambda x)}{N(z,x)} \leq \lambda^{-\delta(z)} \quad    0<\lambda< 1. 
	\end{equation}
Let 
\begin{equation}\label{ff}
	N^* (x,\lambda) := \limsup_{u\to \infty} \frac{N(x,\lambda u)}{N(x,u)}, \quad
	N_* (x,\lambda) := \liminf_{u\to \infty} \frac{N(x,\lambda u)}{N(x,u)}.
\end{equation}
We  assume that 
\begin{equation}\tag{\bfseries N2}\label{N2}
	\text{The functions $N_*(x,\lambda)$ and $N^*(x,\lambda)$ are continuous in $x$ for all $\lambda\geq 1$ and $|x|\geq x_0$}. 
\end{equation}
By \eqref{N1} and \eqref{N2}  we have 
\begin{equation}\label{N11} 
\frac{1}{\lambda^\delta}\leq \inf_z   N_*(z,\lambda) \leq \sup_z  N^*(z,\lambda)  \leq \frac{1}{\lambda^\sigma}. 
\end{equation}
On the other hand, \eqref{N1} implies the following estimates on $N(x,x)$: there exists $N_\sigma, \,N_\delta>0$,  such that
\begin{equation}\label{N111}
	  \frac{ N_\delta }{x^{\delta}}\leq  N(x,x)\leq \frac{ N_\sigma}{x^{\sigma}}, \quad |x|\geq c. 
  \end{equation}
We need to set some additional notation. Let 
\begin{equation}\label{N}
	N(x): = N(x,x), \quad |x|\geq c, 
\end{equation}
	\begin{equation}\label{Nmax}
	N_{max}:= \max_{|x|\geq 1}N(x,1). 
\end{equation} 
\begin{equation}\label{C(x)}
		\nu_{small}(x) :=  \int_{|u|\leq 1} |u|^2 \nu(x,du), \quad |x|\geq 1,  
\end{equation}
	\begin{equation}\label{Cinf}
	\nu_{small} :=  \sup_{|x|\geq 1}\nu_{small}(x) < \infty.  
	\end{equation}
	Denote 
	\begin{equation}\label{C10}
		\begin{cases}
	C^{(1)}:= 2\I_{\delta=\sigma}\sum_{k=1}^\infty C_p^{2k} \left(\frac{N_{\delta} 2k}{2k-\delta} - \frac{N_\delta (2k-p)}{2k-p+\delta}\right)  + \frac{2p N_\sigma}{\sigma-p},\quad  &1<\sigma\leq \delta<2,\\
			C^{(2)}:=2p(p-1)N_{\delta}, \quad  &\sigma= \delta=2 \\
			C^{(3)}:= 
			\frac{p(p-1)}{2} \left(\nu_{small} +  2N_{max} +  \frac{4N_\delta}{\delta-2}\right)
			+ \frac{2p N_\sigma\I_{\sigma =2}}{\sigma -p} , \quad  &\sigma \geq 2, \, \delta \neq 2, \\
			C^{(4)}:=  \frac{2pN_\sigma}{\sigma-p},\quad  &1<\sigma <2, \, \delta\geq 2, 
		\end{cases} 
	\end{equation}
	where $N_\sigma$, $N_\delta$ and $N_{max}$ are defined in \eqref{N111} and  \eqref{Nmax}, respectivelty. Finally, define the following scaling functions: 
	
	\begin{align}\label{phi1}
		\phi^{(1)}(x)&:= |x|^{p-\sigma}, &  &1<\sigma\leq  \delta<2;\\ \label{phi2}
		\phi^{(2)}(x)&:=  |x|^{p-2}\ln(1+ |x|), &  & \sigma= \delta= 2; \\  \label{phi3}
		\phi^{(3)}(x)&:=   |x|^{p-2}, &  & \sigma \geq 2,\, \delta>2, \\
		\label{phi4}
		\phi^{(4)}(x)&:= \phi^{(1)}(x),&  &\sigma <2, \, \delta \geq 2. 
	\end{align} 
 In what follows $f$ and $\psi$ are related by \eqref{F}, and $f$ varies from theorem to theorem. \normal

	\begin{theorem}\label{t1}
		Suppose that  $p\in (1,\sigma)$,  and $f(x)\geq 1$, $x\in \real$, and one of the conditions below holds true with  some $f$ (which may vary from case to case):
		\begin{equation}\label{t1e1}
			\limsup_{|x|\to \infty} \left( \frac{p a(x)\sign(x)|x|^{p-1} + f(|x|^p\normal)}{\phi^{(i)}(x)} + C^{(i)}\right)<0, \quad 1\leq i\leq 4, 
		\end{equation}
		where $\phi^{(i)}(x)$ and $C^{(i)}$ are defined in \eqref{phi1}--\eqref{phi3} and \eqref{C10}, respectively. 
		Then the respective process $X$  is  $\psi$-ergodic. 
	\end{theorem}
	In order to write the ergodicity condition in a simpler form, we impose one more assumption. Assume  the following condition on the drift $a(x)$ (see also \cite{Ku17}, cf.~(3.3.4)): 
	\begin{equation}\label{a-gr}
		\limsup_{|x|\to \infty} \frac{a(x)\sign(x)}{|x|^\kappa}\leq - A_\kappa\in (-\infty,0).
	\end{equation}
	
	\begin{corollary}\label{c2}
		Suppose that the assumptions of Theorem~1 hold true and  
		\begin{equation}\label{bal1} 
			\kappa+\min(\sigma,2) >1. 
		\end{equation}
		Then the process $X$  is $\psi$-ergodic, where $\psi$ is related  to  $f(x) = C x^{1+(\kappa-1)/p}$, $C>0$,   for $\kappa \in [-1, 1)$, and $f(x) =   C t$ for $\kappa \geq 1$,  by \eqref{F2}.  
	\end{corollary}

\begin{remark}
	\begin{itemize}
		\item[i)]
	Observe that  for $t\geq 1 $ we have 
	\begin{equation}\label{F10}
	F(t) = 	\begin{cases}
		\frac{\ln{t}}{C}, & \kappa \geq  1, \\
			\frac{p}{C(1-\kappa)}\left( t^{\frac{1-\kappa}{p}}-1\right), & \kappa\in [-1,1).
		\end{cases}
	\end{equation}
Consequently, for any $\gamma>0$ 
 \begin{equation}
	\psi(t) = 
	\begin{cases}
		C^{-1}e^{-C\gamma t}, & \kappa \geq  1, \\
		C^{-1}\left(\frac{C\gamma (1-\kappa)}{p}t+1 \right)^{1-\frac{p}{1-\kappa}}, & \kappa\in [-1,1).
	\end{cases}
\end{equation}
Thus, the rate of convergence in \eqref{erg2} is exponential, if $\kappa\geq 1$, and is of order $t^{\frac{p}{1-\kappa}-1}$,  if $\kappa\in [-1,1)$, provided that \eqref{bal1} holds true.

\item[ii)] Note that for $\kappa >1$ the function $\frac{p}{C(1-\kappa)}\left( t^{\frac{1-\kappa}{p}}-1\right)$  is bounded. 
	\end{itemize}   
\end{remark}
	
	Consider now the light-tail case. We assume that there exist $\alpha>0$, $\zeta\in (-1,0]$ such that 
	\begin{equation}\label{ExAl} 
		\nu_{\alpha,\zeta, large} (x)  :=  	\int_{|u|\geq 1} e^{\alpha|u|^{1+\zeta}}\nu(x,du)<\infty \quad \text{for any $x\in \real$}. 
	\end{equation}
	Suppose that
	\begin{equation}\label{numax}
		\nu_{\alpha,\zeta, large} := \sup_{|x|\geq 1} 	\nu_{\alpha,\zeta, large} (x)<\infty. 
	\end{equation}
	Let 
	\begin{equation}\label{phi5c5}
	\phi^{(5)}(x) = |x|^{\kappa+\zeta} e^{\beta|x|}, \qquad C^{(5)}:= 2^{-1}\beta^2 (1+\zeta)^2 \left(  e^\beta \nu_{small}+ c_0  \nu_{\alpha,\zeta,large}\right), 
	\end{equation}
where $c_0$ is such that $x^2e^{\beta x^{1+\zeta}}\leq c_0 e^{\alpha x^{1+\zeta}}$, $x>1$. 

\begin{theorem}\label{t2} 
	Suppose that $ \beta>0$, $\zeta\in (-1,0]$ are  such that \eqref{numax} is satisfied.  Suppose that $f$ satisfies the inequality
\begin{equation}\label{erg-exp02}
	\limsup_{|x|\to \infty} \left(\frac{ \beta  (1+\zeta) a (x)\sign(x)}{|x|^\kappa} + \frac{f(e^{\beta|x|^{1+\zeta}})}{\phi^{(5)}(x)}+ C^{(5)}\right)<0.
\end{equation} 
Then the process $X$  is $\psi$-ergodic. 
\end{theorem} 
In order to write this result in a more convenient  form assume \eqref{a-gr}.  
\begin{corollary}\label{c4}
	\begin{itemize}
	\item[i)] Suppose that   $\zeta=0$ and  \eqref{a-gr} holds with $\kappa\geq 0$.
	Then for $\kappa>0$ the process $X$ is $\psi$-ergodic with    $f(x) = C x$, where $C>0$ is arbitrary.
	 
\noindent And for $\kappa =0$ the process $X$  is  $\psi$-ergodic with the same $f$ as above, provided that 
\begin{equation}\label{erg-exp1}
	- \beta A_0 + C + C^{(5)}<0.  
\end{equation}

\item[ii)]  Suppose that $\zeta\in  (-1,\kappa]$  and  \eqref{a-gr} holds   with $\kappa \in (-1,0) $ and  $ f(x) =Cx\left(\frac{1}{\beta} \ln{x}\right)^{\frac{\kappa+\zeta}{1+\zeta}}$.   Then the process $X$  is $\psi$-ergodic,  provided that 
\begin{equation}\label{erg-exp2}
	- \beta (1+ \zeta) A_\kappa + C + C^{(5)}<0.  
\end{equation} 
\end{itemize}
\end{corollary}

  It is possible to build the corresponding function $\psi$ from \eqref{erg2}. For $f(x) = Cx$ we have
$$
F(t) = C^{-1}\ln{t}, \qquad \psi(t) = C^{-1}e^{-C\gamma t}, 
$$
and  for $f(x)= C\left(\beta^{-1}\ln{x}\right)^{\frac{\kappa+\zeta}{1+\zeta}} x$ we have 
$$
F(t) = \frac{1+\zeta}{C(1-\kappa)} \beta^{\frac{\kappa+\zeta}{1+\zeta}}\left(\ln{t}\right)^{\frac{1-\kappa}{1+\zeta}},
$$
and 
$$
\psi(t) = C^{\frac{1+\zeta}{\kappa-1}} \beta^{\frac{\kappa+\zeta}{1-\kappa}} \left(\frac{\gamma (1-\kappa)}{1+\zeta}t\right)^{\frac{\kappa+\zeta}{\kappa-1}} e^{-\left(C \beta^{-\frac{\kappa+\zeta}{1+\zeta}}\frac{\gamma(1-\kappa)}{1+\zeta}t\right)^{\frac{1+\zeta}{1-\kappa}}}.
$$
\begin{remark}\label{rem-zeta}
It is possible to choose $\zeta$ optimally. Note that $\kappa<0$ and $1+\zeta>0$, thus the power in the exponent is positive. It will be maximal, if we choose $\zeta$ maximal possible, e.g. $\zeta=\kappa$. Note also, that  $\frac{\kappa+\zeta}{\kappa-1}>0$, thus the decay is determined by the exponent. 
\end{remark}

\section{Proofs}\label{Proofs}

We begin with the proof of Theorem~\ref{t1}. 

Recall the following series decomposition. 	Let $x,y$ be real, $p\in \real$. If $|x|>|y|$, then 
\begin{equation}\label{bin2}
	(x+y)^{p}= \sum_{k=0}^\infty C_p^k x^{p-k}y^k. 
\end{equation}

Let $\phi: \real\to \real_+$, $\phi\in C^2(\real)$, be defined as follows:
\begin{equation}\label{phi}
	\phi(x) =|x|\quad \text{ for $|x|>1$\,\, and} \qquad \phi(x)\leq |x|\qquad\text{ for $|x|\leq 1$.}
\end{equation}

Define 
\begin{equation}\label{Vp}
	V (x)= \phi^p(|x|), \quad p>0, \quad x\in \real;
\end{equation}
note that in both cases $V\in \mathcal{D}_0$. 

The proof of Theorem~\ref{t1}  relies from the following  lemma. 

	\begin{lemma}\label{lem1}
		Let  $p\in (1,\sigma)$ in the definition \eqref{Vp} of $V$. Then 
		\begin{equation}\label{limV}
			\limsup_{x\to \infty}  \frac{\LL_0 V (x) }{\phi^{(i)}(x)}\leq C^{(i)}, \quad 1\leq i\leq 4, 
		\end{equation}
		where the functions $\phi^{(i)}$ are defined for the respective $\sigma$ and $\delta$ in \eqref{phi1}--\eqref{phi4}.  
		\end{lemma}
\begin{proof}
  \emph{ Case (1).}  Note that in this part we assumed that $\delta<2$.
	\begin{align*}
		\LL_0  V (x)= \left(\int_{|u|>1}  + \int_{|u|\leq 1} \right) \left(V (x+u)- V(x) - V'(x)\I_{|u|\leq 1} \right) \nu(x,du)= I_1(x)+I_2(x). 
	\end{align*}
For $I_2(x)$ we have   for $x$  large enough 
\begin{equation}\label{I2}
	\begin{split}
I_2(x) &=   \frac{p(p-1)}{2} \int_{|u|\leq 1}  (u(1-\zeta))^2 \left[\int_0^1 (x+ \zeta u )^{p-2} d\zeta\right] \nu(x,du)\leq \frac{C(x)}{x^{2-p}}. 
\end{split}
\end{equation}
Since $\delta<2$,  by \eqref{N11} we get 
\begin{equation}\label{del1}
	\lim_{x\to \infty} \frac{I_2(x)}{\phi^{(1)}(x) }=0. 
\end{equation}
Let us estimate $I_1(x)$. Let 
$$
\tilde{I}_1 (x) := \int_{|u|>1} \left(|x+u|^p - |x|^p\right)
\nu(x,du), 
$$
and split  
\begin{align*}
	I_1(x) \pm \tilde{I}_1(x) &= \left(\int_{-\infty}^{-x} +  \int_{x}^\infty + \int_{-x}^{-1}+ \int_1^x \right)  |x+u|^p\nu(x,du)  -|x|^p \int_{|u|>1} \nu(x,du)\\
	& \qquad + \int_{|u|>1, |x+u|<1} ((\phi(|x+u|))^p- |x+u|^p)\nu(x,du)\\
	&= I_{11}(x)+ I_{12}(x)+ I_{13}(x) + I_{14}(x)+ I_{15}+I_{16}. 
\end{align*}
In order to estimate $I_{1k}(x)$, $1\leq k\leq 5$, we use  \eqref{bin2}. 

For $I_{16}(x)$ we have 
\begin{align*}
	|I_{16}(x)|&\leq 2 \int_{|u|>1, |x+u|<1} \nu(x,du)\leq 
	2  \nu(x,B(x,1)), 
\end{align*}
implying 
$$
\limsup_{x\to \infty} \frac{|I_{16}(x)|}{ x^{p-\sigma} }\leq \limsup_{x\to \infty} \frac{2 N(x,x-1)}{x^{p-\sigma}})=0. 
$$

For $I_{11}(x)$ we have  due to the symmetry of $\nu(x,du)$ in $u$
\begin{align*}
	I_{11}(x)& = x^p \int_x^\infty \left(-1+\frac{u}{x}\right)^p\nu(x,du)  =  x^p \int_x^\infty \sum_{k=0}^\infty C_p^k \left(\frac{u}{x}\right)^{p-k} (-1)^k  \nu(x,du). 
\end{align*}
where in the second equality  we used \eqref{bin2}.  Similarly, 
$$
I_{12}(x)= x^p \int_x^\infty \sum_{k=0}^\infty C_p^k \left(\frac{u}{x}\right)^{p-k}   \nu(x,du), 
$$
implying  after interchanging the sum and in the integral, the equality  
$$
I_{11}(x)+ I_{12}(x) = 2x^p  \sum_{k=0}^\infty
\int_x^\infty  C_p^{2k}\left(\frac{u}{x}\right)^{p-2k}   \nu(x,du),
$$
Since  $p<\sigma$, 
$$
u^{p-2k} N(x,u) \leq \frac{ N_\sigma N(x,1)}{u^{2k+\sigma-p}} \to 0, \quad u\to \infty, 
$$
 for any fixed $x$, $|x|\geq c$, and  thus 
$$
\int_x^\infty u^{p-2k} \nu(x,du)= x^{p-2k} N(x)+ (p-2k)\int_x^\infty u^{p-2k-1}N(x,u)du, \quad k\geq 0. 
$$
Therefore, interchanging the series and  the integral, we get 
\begin{align*}
	I_{11}(x)+ I_{12}(x) &=2 x^p  \sum_{k=0}^\infty C_p^{2k}  x^{2k-p} \left( x^{p-2k} N(x)+ (p-2k)\int_x^\infty \frac{N(x,u)}{u^{2k+1-p}} du \right)\\
	&= 2x^p N(x)\sum_{k=0}^\infty C_p^{2k} \left( 1+ (p-2k) x^{2k-p}\int_x^\infty \frac{1}{u^{2k+1-p}}\frac{N(x,u)}{N(x)} du
	\right)\\
	&= 2x^p N(x) \left(\sum_{k=0}^\infty C_p^{2k} + \sum_{k=0}^\infty  C_p^{2k}  (p-2k)\int_1^\infty \frac{1}{u^{2k+1-p}}\frac{N(x,xu)}{N(x)} du\right). 
\end{align*}
Consider now $I_{13}(x)$. As above, due to the symmetry of $\nu(x,du)$ and \eqref{bin2}, 
\begin{align*}
	I_{13}& = x^p \int_1^x  \left( 1- \frac{u}{x} \right)^p \nu(x,du)= x^p \int_1^x \sum_{k=0}^\infty  C_p^k \left( \frac{u}{x}\right)^k (-1)^k \nu(x,du). 
\end{align*}
Analogounsly, 
\begin{align*}
	I_{14}& = x^p \int_1^x  \left( 1+ \frac{u}{x} \right)^p \nu(x,du)= x^p \int_1^x \sum_{k=0}^\infty  C_p^k \left( \frac{u}{x}\right)^k  \nu(x,du),
\end{align*}
implying that 
$$
I_{13}+ I_{14} = 2x^p \int_1^x \sum_{k=0}^\infty  C_p^k \left( \frac{u}{x}\right)^{2k}  \nu(x,du)= 
 2x^p \int_1^x \sum_{k=1}^\infty  C_p^k \left( \frac{u}{x}\right)^{2k}  \nu(x,du)+ 2x^p\int_1^x \nu(x,du), 
$$
which leads to 
$$
I_{13}+ I_{14} + I_{15}= 2x^p \int_1^x \sum_{k=1}^\infty  C_p^k \left( \frac{u}{x}\right)^{2k}  \nu(x,du) - 2 x^p N(x). 
$$
Using integration by parts, we get 
\begin{align*}
	I_{13}+I_{14}+I_{15} & =  2x^p  
\sum_{k=1}^\infty  C_p^{2k} x^{-2k}\left( -x^{2k}  N(x)+ N(x,1) + 2k \int_1^x u^{2k-1} N(x,u)du \right)- 2 x^p N(x)\\
&= 2x^p N(x) \left(- \sum_{k=0}^\infty C_p^{2k} + \sum_{k=1}^\infty  C_p^{2k}  \left( \frac{N(x,1)}{N(x)x^{2k} } + \frac{2k}{x^{2k}} \int_1^x u^{2k-1} \frac{N(x,u)}{N(x)}du\right) \right)\\
&= 2x^p N(x) \left(- \sum_{k=0}^\infty C_p^{2k} + \sum_{k=1}^\infty  C_p^{2k}  \left( \frac{N(x,1)}{N(x)x^{2k} } + 2k \int^1_{1/x} u^{2k-1} \frac{N(x,xu)}{N(x)}du\right) \right). 
\end{align*}
Thus,  
\begin{equation}\label{I1}
	\begin{split}
	I_1(x)&\leq  2x^p \sum_{k=1}^\infty C_p^{2k}  \left(\frac{N(x,1)}{x^{2k} }+  (p-2k)\int_1^\infty \frac{N(x,xu)}{u^{2k+1-p}} du+
	2k \int_{1/x}^1 u^{2k-1} N(x,xu)du  \right)\\
	&\qquad  + 2 x^p p \int_1^\infty u^{p-1} N(x,xu) du. \normal
	\end{split}
\end{equation}


Note that for $p\in (1,2)$ we have $C_p^{2k}>0$. Then 
\begin{equation*}\label{I01}
	\begin{split}
\limsup_{x\to \infty} \frac{I_1(x)}{\phi^{(1)}(x)} &\leq 
\limsup_{x\to \infty}2   \sum_{k=1}^\infty  C_p^{2k} \left(  (p-2k)\int_1^\infty \frac{ x^\delta N(x,xu)}{u^{2k+1-p}} du+
2k \int_{1/x}^1 u^{2k-1} x^\delta N(x,xu)du  \right)\\
& \qquad   +2p N_\sigma\int_1^\infty u^{p-\sigma-1}du. 
\end{split}
\end{equation*}. 
   If $\sigma<\delta<2$, then all the terms in the sum $\sum_{k=1}^\infty (\dots)$  vanish, and we have 
\begin{equation}
	\limsup_{x\to \infty} \frac{I_1(x)}{\phi^{(1)}(x)}\leq 
	\frac{2pN_\sigma}{\sigma-p}. 
\end{equation}

If $\sigma=\delta<2$, we have 
\begin{equation}
	\limsup_{x\to \infty} \frac{I_1(x)}{\phi^{(1)}(x)}\leq 
	2\sum_{k=1}^\infty C_p^{2k} \left(\frac{N_{\delta} 2k}{2k-\delta} - \frac{N_\delta (2k-p)}{2k-p+\delta}\right)  + \frac{2pN_\sigma}{ \sigma-p}.
\end{equation}
Thus, we get 
$$
\limsup_{x\to \infty} 	\frac{I_1(x) }{\phi^{(1)}(x)} \leq  C^{(1)}. 
$$
which implies \eqref{limV} for $i=1$. 

\emph{ Case (2).}  If $\sigma=2=\delta$ we have 
\begin{equation}\label{sig1}
	\lim_{x\to \infty} \frac{I_2(x)}{ \phi^{(2)}(x)}=0.  \normal
\end{equation}
Consider now $I_1(x)$. Due to the $\ln x$ in the definition of $\phi^{(2)}$, all terms vanish except the following for $k=1$: 
$$
\limsup_{x\to \infty} \frac{x^{2}}{\ln x} \int_{1/x}^1 u^{2k-1} N(x,xu)du \leq N_{\delta}. 
$$
Therefore, in this case we have 
\begin{equation}
	\limsup_{x\to \infty} \frac{I_1(x)}{\phi^{(2)}(x)}\leq 
	4 C_p^2 N_{\delta}= 2p(p-1)N_{\delta}. 
\end{equation}

\emph{ Case (3).} Recall that in this case $\phi^{(3)}(x) = x^{p-2}$. 
Then 
\begin{equation}\label{I02}
	\limsup_{x\to \infty}  \frac{I_2(x)}{\phi^{(3)}(x)}\leq   	\limsup_{x\to \infty}  \frac{p(p-1)\nu_{small}(x)x^{p-2}}{2\phi^{(3)}(x)} = \frac{p(p-1)\nu_{small}}{2}.
\end{equation}
Consider now $I_1(x)$.
Since $\delta>2$, we have 
$$
\limsup_{x\to \infty} x^2 \int_1^\infty \frac{N(x,xu)}{u^{2k+1-p}}du =0. 
$$
If  $\delta \neq 2m$ for some $m\geq 2$, then
\begin{align*}
	\limsup_{x\to \infty} x^2 \int_{1/x}^1 u^{2k-1} N(x,xu)du & \leq  \frac{N_\delta}{2k-\delta}\limsup_{x\to \infty} 
	\left( x^{2-\delta}- x^{2-2k}\right)=\frac{N_{\delta}}{\delta-2}. 
\end{align*}

If $\delta = 2m$, we  also have 
$$
\limsup_{x\to \infty} x^2 \int_{1/x}^1 u^{2k-1} N(x,xu)du \leq N_\delta\lim_{x\to \infty} x^{2-2m}\ln{x}=0. 
$$

Thus, from above 
\begin{align*}
	\limsup_{x\to \infty} \frac{I_1(x)}{\phi^{(3)}(x)}  &\leq  p(p-1)\left(N_{max}+ \frac{2N_{\delta}}{\delta - 2}\right) \normal + \limsup_{x\to \infty}  2p N_\sigma x^{2-\sigma}\int_1^\infty u^{p-1-\sigma }\\
	&  = p(p-1)\left( N_{max} + \frac{2N_{\delta}}{\delta - 2}\right) \normal
	+ \frac{2p N_\sigma\I_{\sigma =2}}{\sigma -p}. 
\end{align*}

\emph{ Case (4).} Finally, from above it follows that in the case $\sigma<2$, $\delta\geq 2$ we have 
$$
	\limsup_{x\to \infty} \frac{I_2(x)}{\phi^{(1)}(x)}\leq 0 
$$
and
$$
   \limsup_{x\to \infty} \frac{I_1(x)}{\phi^{(1)}(x)}\leq \frac{2pN_\sigma}{\sigma-p}.
$$
\normal
\end{proof}

\begin{proof}[Proof of Corollary~\ref{c2}]
 Take  $f(|x|^p)\equiv f^{(i)}(|x|^p):= C |x|^{p(1+(\kappa-1)/p)}$  for $1\leq i\leq 4$, where $C<pA_\kappa$. 
  For $i=1$  or $i=4$ the inequality  \eqref{t1e1}  holds if $\kappa+\sigma>1$; in this case the $\limsup$  in \eqref{t1e1}  is  equal to $-\infty$. 
 For $i=3$ the inequality  \eqref{t1e1}  holds if $\kappa+1>0$, which however is true if \eqref{bal1} is satisfied,  because we initially have $\sigma>1$. Finally, for $i=2$  the inequality \eqref{t1e1} is satisfied, provided that $\frac{|x|^{\kappa+1}}{\ln |x|} \to  \infty$ as $|x|\to \infty$.  The only critical case $\kappa = -1$ is excluded, because for $i=2$ we have $\sigma=2$, and thus \eqref{bal1} is violated.  
\end{proof}

\begin{proof}[Proof of Theorem~\ref{t2}] 
 	Consider $V(x) = e^{\beta \phi^{1+\zeta}(x)}$ for some $ \zeta \in (-1,0]$, $\beta<\alpha$, where $\phi(x)$ is defined in \eqref{phi}.  Since we assumed \eqref{ExAl}, we  modify a bit the operator $\LL$: 
$$
\LL V(x) = \left(a(x) + \int_{|u|\geq 1} u \nu(x,du)\right) V'(x) + \int_\real \left( V(x+u)  - V(x) - V'(x) u \right) \nu(x,du). 
$$ 
Recall that we assumed that $\nu(x,du)$ is symmetric in $u$; therefore, the integral term in the coefficient near $V'(x)$ is 0.\normal

	First we show  that 
	\begin{equation}\label{exp2-1}
		\limsup_{|x|\to \infty} \frac{\LL_0 V(x)}{\phi^{(5)}(x)} \leq C^{(5)},  
	\end{equation}
	where $\phi^{(5)}$ and $C^{(5)}$ are defined in  \eqref{phi5c5},  and $\LL_0$ now is  defined by 
	 
	$$
	 \LL_0 V(x) = \int_\real \left( V(x+u)  - V(x) - V'(x) u \right) \nu(x,du). 
	$$ 
	
	Assume that $x>1$ is big enough, and let for $g(u)= e^{\beta |u|^{1+\zeta}}$
	\begin{equation}
		I_1 (x) = \int_{|u|\geq 1} \left( g(x+u)- g(x) - g'(x) u\right)  \nu(x,du)
	\end{equation}
 	and split
	\begin{align*}
		\mathcal{L}_0  V(x) & = \left(\int_{|u|< 1}+ \int_{|u|\geq  1}  \right)\left(V(x+u) -V(x) - V'(x) u \right)\nu(x,du)\pm I_1 (x)\\
		& = \left( \int_{|u|<  1} +\int_{1\leq |u| \le \epsilon x }+ \int_{|u|>\epsilon x} \right) \left( g(x+u)- g(x) - g'(x) u\right)  \nu(x,du) \\
		&\quad  + \int_{|u|>  \epsilon x } \left(V(x+u) - e^{\beta|x+u|^{ 1+  \zeta}}\right) \nu(x,du) \\ 
		& =  J_1(x) + J_2(x) +J_3 + J_4(x), 
	\end{align*}
where $\epsilon>0$	is such that $\beta+\epsilon<\alpha$. Using the Taylor formula for $g(u)$  we get for some $\theta\in (0,1)$
\begin{equation}\label{Taylor1}
	\begin{split}
|g(x+u)- g(x) - g'(x) u|& \leq 2^{-1}|g''(x+\theta u )| u^2 \\
& =  2^{-1} \beta (1+\zeta) (x+ u \theta)^{2\zeta}  \left( \beta  (1+\zeta) +  \zeta  (x+\theta u)^{-1-\zeta}\right)\int_{|u|\leq 1}u^2 \nu(x,du), 
\end{split}
	\end{equation}
implying   for $|u|\leq 1$
	\begin{equation}\label{Ex2J1}
		\begin{split}
		J_1(x) &  \leq 2^{-1} \beta (1+\zeta) (x-1)^{2\zeta}V(x+1) \left( \beta  (1+\zeta) +  \zeta  (x-1)^{-1-\zeta}\right)\int_{|u|\leq 1}u^2 \nu(x,du)\\
		& \leq  C_1 (x-1)^{2\zeta}V(x+1) \left(1+o(1)\right).   
		\end{split}
	\end{equation}
where  in the last line we used that $|x+u|^{1+\zeta} \leq x^{1+\zeta} + |u|^{1+\zeta} \leq  x^{1+\zeta} +1$ (cf. \eqref{Cinf})
$$
C_1= 2^{-1} e^\beta \beta^2 (1+\zeta)^2\nu_{small}. 
$$
The argument for $J_2$ is similar. Using \eqref{Taylor1} and that $1<|u|\leq \epsilon x$  we derive
\begin{align*}
\left| g(x+u)- g(x) - g'(x) u \right|
& \leq 2^{-1}(1-\epsilon)^{2\zeta}\beta^2 (1+\zeta)^2  x^{2\zeta} V(x)e^{\beta|u|^{1+\zeta}} u^2 (1+o(1)),  
\end{align*}
implying that
\begin{equation}\label{J2-exp}
	J_2(x) \leq C_2 x^{2\zeta} V(x)(1+o(1)),  
\end{equation}
where 
$$
C_2:= 2^{-1}\beta^2 (1+\zeta)^2  c_0 \nu_{\alpha,\zeta,large},
$$
where $c_0$ is such that $x^2e^{\beta x^{1+\zeta}}\leq c_0 e^{\alpha x^{1+\zeta}}$, $x>1$, and we used that $(1-\epsilon)^{2\zeta}\leq 1$. For $J_3$ we have 

\begin{align*}
	J_3(x) & \leq \int_{|u|\geq \epsilon x} e^{\beta |x + u|^{1+\zeta}}\nu(x,du) 
	\leq V(x) \int_{|u|\geq \epsilon x} e^{- (\alpha - \beta)  |u|^{1+\zeta}+ \alpha |u|^{1+\zeta}}\nu(x,du)\\
	& \leq V(x) e^{-(\alpha -\beta)(\epsilon x)^{1+\zeta}}\int_{|u|\geq 1} e^{\alpha |u|^{1+\zeta}} \nu(x, du)\\
	& \leq  \nu_{\alpha,\zeta,large} V(x) e^{-(\alpha -\beta)(\epsilon x)^{1+\zeta}}. 
\end{align*}
Finally,
\begin{align*}
	J_4 (x) & \leq  \int_{|u|\geq \epsilon x} e^{\beta |x+u|^{1+\zeta}}\nu(x,du)\leq \nu_{\alpha,\zeta,large} V(x) e^{-(\alpha -\beta)(\epsilon x)^{1+\zeta}}. 
\end{align*}
Thus, we have \eqref{exp2-1} with 
$$
C^{(5)}  = C_1+C_2. 
$$
Consider now the drift term.  We have  
	\begin{equation}
		\limsup_{x\to \infty}  \frac{ a(x) V'(x) }{\phi^{(5)}}= 	\limsup_{x\to \infty} \frac{ a(x) \beta (1+\zeta)   x^\zeta V(x)}{\phi^{(5)}(x)} \leq - \beta (1+\zeta)  A_{\kappa }. 
	\end{equation}
Summarizing, we get \eqref{erg-exp02}.  The restriction $\zeta\in (-1,\kappa]$  comes from the following argument. The orders of the drift, the function $f(V(x))$  and of the integral terms are, respectively, $x^{\kappa +\zeta} V(x)$,  $x^{\kappa +\zeta} V(x)$  and $x^{2\zeta} V(x)$. Since for the integral term we have the positive bound and the function $f(V(x))$ is positive, we need the drift to dominate. Therefore, we need the restriction $\kappa+\zeta> 2 \zeta$, which implies the necessary  bound for $\zeta$.
\end{proof}

The proof of Corollary~\ref{c4} follows in the same way as that of Corollary~\ref{c2}. 
 
\section{Application}
Consider one simple application of Theorems~\ref{t1} and \ref{t2}.

Suppose we want to estimate the value of the functional $ \pi^x(u):=  \Ee^\pi u(X)$,  where $X$ is a random variable with probability distribution $\pi$. Knowing  the order of convergence, we can estimate $\pi^x(u)$ with a given precision choosing $t$ large enough according to \eqref{erg2}: 
\begin{equation}
	|T_t u(x) - \pi^x(u)|\leq \sup_{z\in \real}|u(z)|	\|P_t(x,\cdot)-\pi(\cdot)\|_{TV} \leq  C_u C_{erg}(x) \psi(t).
\end{equation}

In \cite{BS09} it was suggested to approximate $X_t$ by a Markov chain $(Y_n(t))_{n\geq 1}$. 
The advantage of this approach is that the chain $(Y_n(t))_{n\geq 1}$ is relatively easy to simulate.  Let us briefly explain the construction. 

Suppose that $x$ in \eqref{L1} is fixed and define the operator $L$ from \eqref{L1} by $L^x$. Then $(L^x, C_0^\infty(\real))$ extends to the generator of a L\'evy process $Y^x$ with the  characteristic function 
$$
\Ee e^{i \xi Y_t^x} = e^{- t q(x,\xi)}, \quad \xi \in \real, 
$$
where  (see the monograph \cite{Ja01} for the detailed explanations)
$$
q(x,\xi) := -i a(x) \xi + \int_{\real \backslash\{0\}}  \left( 1-e^{i \xi u }+ i\xi  u\I_{|u|\leq 1} \right)\nu(x,du). 
$$
Define the family of probability measures  $(\mu_{x,\frac{1}{n}})_{n\geq 1}$ by 
$$
	\int_\real  e^{i\xi y} \mu_{x,\frac{1}{n}} (dy)= e^{i \xi x-\frac{1}{n} q(x,\xi) }. 
$$
Then we define a Markov chain $(Y_n(k))_{n\geq 1}$  with the transition kernel $(\mu_{x,\frac{1}{n}})_{n\geq 1}$, and  let
$$
W_{\frac1n} u(x) := \int_\real u(y)\mu_{x,\frac{1}{n}}(dy). 
$$
Then (cf. \cite{BS09})
\begin{equation}
	\sup_{t\in [a,b]} \| W_{\frac1n}^{[nt]} u - T_t u \|_{\infty}= 0, 
\end{equation}
where $\|u\|_\infty := \sup_x |u(x)|$, and the rate of convergence is of order $n^{-\frac12}$, which follows from Lemma~6.4 and  Theorems~6.1 and 6.5 \cite{EK86}.

Note that $W_{\frac1n}^{[nt]} u (x) = \Ee^x u(Y_n ([nt]))$, and finally applying the Monte-Carlo method, we derive 
\begin{equation}
	\pi^x(u)\approx \frac{1}{N} \sum_{i=1}^N u\left(Y_n^{(i)} ([nt])\right), 
\end{equation}
where 
 $Y_n^{(i)} ([nt])$, $1\leq i \leq N$ are independent copies of $Y_n([nt])$, and $Y_n$ starts at $x$.  Thus, the approach reduces to the simulation of the Markov chain with a given kernel. Some particular examples were treated in \cite{B10}, but the methodology of simulation of a general  L\'evy process heavily depends on the structure of the kernel $\nu(x_0,du)$  (here $x_0$ is fixed). We postpone this problem for future  research.

\end{document}